\documentclass[]{article}
\usepackage{amsmath,amsthm,amssymb}
\title{Knot cobordism and Lee's Perturbation of Khovanov homology}
\author{Zipei Zhuang}

\usepackage{pifont}
\usepackage{indentfirst}
\usepackage{setspace}
\usepackage{mathrsfs}
\usepackage[backref]{hyperref}

\usepackage{subfigure}
\usepackage{graphicx}
\usepackage{import}
\usepackage{xifthen}
\usepackage{pdfpages}
\usepackage{transparent}
\newtheorem{theorem}{Theorem}%
\newtheorem{corollary}{Corollary}%

\begin{document}
	
	\maketitle

	\begin{abstract}
		For a connected cobordism $S$ between two knots $K_1, K_2$ in $S^3$, we establish an inequality involving the number of local maxima, the genus of $S$, and the torsion orders of $Kh_t(K_1), Kh_t(K_2)$, where $Kh_t$ denotes Lee's perturbation of Khovanov homology. This shows that the torsion order gives a lower bound for the band-unlinking number.
	\end{abstract}

	If $K_0, K_1$ are knots in $S^3$, a cobordism from $K_0$ to $K_1$ is a smooth compact oriented surface $F$ embedded in $S^3 \times [0,1]$ with boundary $K_0 \times \{0\} \cup K_1 \times \{1\}$. If $F$ is an annulus, it is called a concordance. Any cobordism $F \subset S^3 \times [0,1]$ is called \emph{ribbon } if the projection from $S^3 \times [0,1]$ to [0,1] restricts to a Morse function on $F$ with only index $0$ and 1 critical points. A knot $K_0$ is \emph{ribbon concordant} to $K_1$ if there exists a ribbon concordance from $K_0$ to $K_1$. 
	
	In the past few years there has been several works on applications of knot homology theories to knot cobordisms. For Khovanov homology $Kh$, it is shown \cite{MR4041014} that a ribbon concordance between two links $(F, K_0, K_1)$ induces an injective map $\phi_F : Kh(K_0) \longrightarrow Kh(K_1)$. Furthermore, Sarkar \cite{MR4092319} showed that if $F$ has $k$ saddles, then $(2x)^dKh_t(K_0) \cong (2x)^dKh_t(K_1)$, where $Kh_t$ is Lee's perturbation of Khovanov homology. Similar results are obtained in the context of knot Floer homology(see \cite{MR4024565} \cite{MR4186142}).
	
	The main result of this paper is the following:
	\begin{theorem} \label{thm1}
	   Suppose that $(F,K_0, K_1)$ is a  connected knot cobordism with $m$ births, $b$ saddles and $M$ deaths. Denote by $\bar{F}$ the cobordism from $K_1$ to $K_0$ obtained by horizontally mirroring $\bar{F}$. Then up to a sign
	   \begin{equation}
	     (2x)^M \cdot \phi_{\bar{F} }  \circ \phi_F 
	     = (2x)^{b-m} \cdot id_{Kh_t(K_0)}	
 \end{equation}
	\end{theorem}
	This is analogous to the main result of \cite{MR4186142}, which is for knot Floer homology.
	
	In particular, if $F$ is a ribbon concordance, then $M=b-m=0$, so $\phi_{\bar{F} }  \circ \phi_F= id_{Kh_t(K_0)}$, $\phi_F$ is injective.
	
	Let  $\mathcal{T}(K)$ be the torsion part of $Kh_t(K)$, and $xo(K)$ be the \emph{extortion order} of $K$, i.e. the minimumal $n$ such that $x^n\mathcal{T}(K)=0$. The same argument as in \cite{MR4186142}   gives the following corollary:

	\begin{corollary}  \label{c1}
	    \begin{equation}
	    	xo(K_0) \leq max\{M, xo(K_1)\}+ 2g(F)
	    \end{equation}
    where $g(F)$ is the genus of $F$.
	\end{corollary}
	
	The \emph{band-unlinking number} $ul_b(K)$ of a knot $K$ is the minimum number of oriented band moves necessary to reduce $K$ to an unlink.  Corollary \ref{c1} furthermore gives
	
	\begin{corollary} \label{c2}
		\begin{equation}
			xo(K) \leq ul_b(K)
		\end{equation}
	\end{corollary}

	\section{Lee's perturbation of Khovanov homology}
	 We will work over a field  $\mathbb{F}$ in which $2$ is invertible. Let \textbf{A} be the 2-dimensional Frobenius algebra over $\mathbb{F}[t]$ with basis \{1,x\}, with the multiplication map $m$, comultiplication map $\triangle$, and counit map $\epsilon$ defined as follows:
	 \begin{equation}
	 	\begin{aligned}
	 	&1 \otimes 1 \stackrel{m}{\longrightarrow} 1   
	 	\quad \quad  1 \otimes x \stackrel{m}{\longrightarrow} x    \\
	 	&x \otimes 1 \stackrel{m}{\longrightarrow} x \quad \quad 
	 	x \otimes x \stackrel{m}{\longrightarrow} t\\
	 	&	1 \stackrel{\triangle}{\longrightarrow} 1 \otimes x + x \otimes 1  \\
	 	&x \stackrel{\triangle}{\longrightarrow} x \otimes x + t \cdot 1 \otimes 1  \\
        & 1 \stackrel{\epsilon}{\longrightarrow} 0   \quad \quad x \stackrel{\epsilon}{\longrightarrow} 1
    \end{aligned}
	 \end{equation}

  Applying this Frobenius algebra to the Kauffman cube of resolutions of a knot diagram, we get a chain complex over $\mathbb{F}[t]$, whose homology is called Lee's perturbation of Khovanov homology \cite{MR2174270}. In particular, setting $t=0$, we get Khovanov's original knot homology; taking $t=1$, it becomes Lee homology \cite{MR2173845}.

	 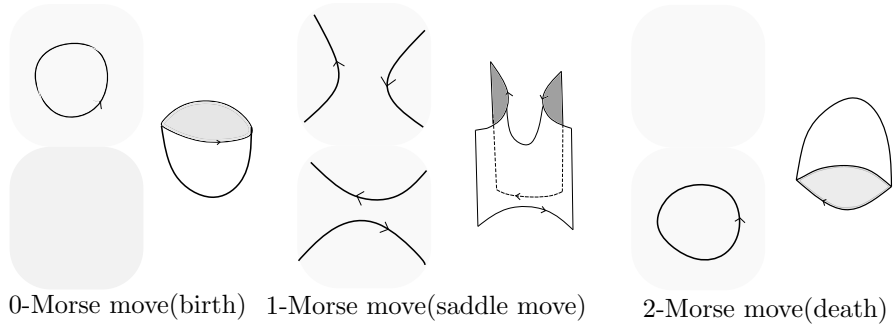
\begin{figure}[h]
	 	\def\svgwidth{\columnwidth}
	 	%% Creator: Inkscape 1.0 (4035a4fb49, 2020-05-01), www.inkscape.org
%% PDF/EPS/PS + LaTeX output extension by Johan Engelen, 2010
%% Accompanies image file 'graph1.pdf' (pdf, eps, ps)
%%
%% To include the image in your LaTeX document, write
%%   \input{<filename>.pdf_tex}
%%  instead of
%%   \includegraphics{<filename>.pdf}
%% To scale the image, write
%%   \def\svgwidth{<desired width>}
%%   \input{<filename>.pdf_tex}
%%  instead of
%%   \includegraphics[width=<desired width>]{<filename>.pdf}
%%
%% Images with a different path to the parent latex file can
%% be accessed with the `import' package (which may need to be
%% installed) using
%%   \usepackage{import}
%% in the preamble, and then including the image with
%%   \import{<path to file>}{<filename>.pdf_tex}
%% Alternatively, one can specify
%%   \graphicspath{{<path to file>/}}
%% 
%% For more information, please see info/svg-inkscape on CTAN:
%%   http://tug.ctan.org/tex-archive/info/svg-inkscape
%%
\begingroup%
  \makeatletter%
  \providecommand\color[2][]{%
    \errmessage{(Inkscape) Color is used for the text in Inkscape, but the package 'color.sty' is not loaded}%
    \renewcommand\color[2][]{}%
  }%
  \providecommand\transparent[1]{%
    \errmessage{(Inkscape) Transparency is used (non-zero) for the text in Inkscape, but the package 'transparent.sty' is not loaded}%
    \renewcommand\transparent[1]{}%
  }%
  \providecommand\rotatebox[2]{#2}%
  \newcommand*\fsize{\dimexpr\f@size pt\relax}%
  \newcommand*\lineheight[1]{\fontsize{\fsize}{#1\fsize}\selectfont}%
  \ifx\svgwidth\undefined%
    \setlength{\unitlength}{841.88976378bp}%
    \ifx\svgscale\undefined%
      \relax%
    \else%
      \setlength{\unitlength}{\unitlength * \real{\svgscale}}%
    \fi%
  \else%
    \setlength{\unitlength}{\svgwidth}%
  \fi%
  \global\let\svgwidth\undefined%
  \global\let\svgscale\undefined%
  \makeatother%
  \begin{picture}(1,0.37037037)%
    \lineheight{1}%
    \setlength\tabcolsep{0pt}%
    \put(0,0){\includegraphics[width=\unitlength,page=1]{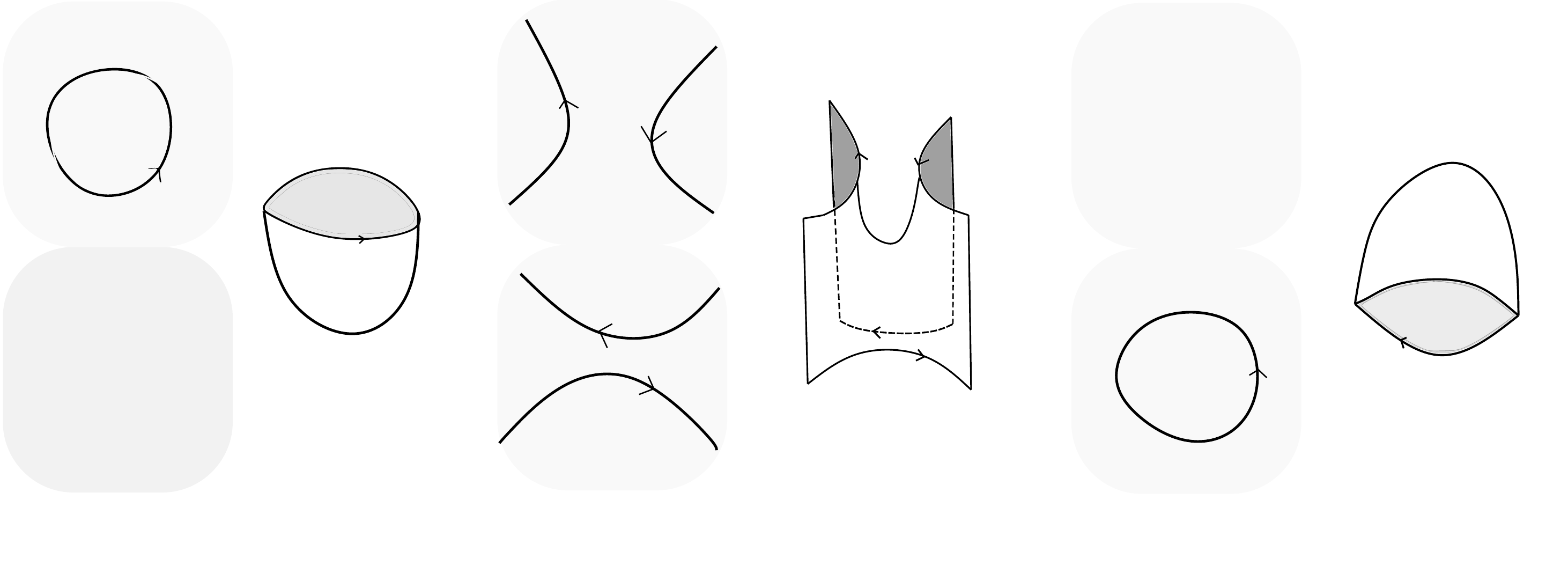}}%
    \put(0.00085426,0.030721){\color[rgb]{0,0,0}\makebox(0,0)[lt]{\lineheight{1.25}\smash{\begin{tabular}[t]{l}0-Morse move(birth)\end{tabular}}}}%
    \put(0.28277021,0.02931516){\color[rgb]{0,0,0}\makebox(0,0)[lt]{\lineheight{1.25}\smash{\begin{tabular}[t]{l}1-Morse move(saddle move)\end{tabular}}}}%
    \put(0.69620311,0.02613355){\color[rgb]{0,0,0}\makebox(0,0)[lt]{\lineheight{1.25}\smash{\begin{tabular}[t]{l}2-Morse move(death)\end{tabular}}}}%
  \end{picture}%
\endgroup%

	 	\caption{The elementary cobordisms} 
	 	\label{graph1}
	 \end{figure}

  A knot cobordism $(F, K_0, K_1)$ induces a map (see \cite{MR1740682} \cite{MR2113903}  \cite{MR2174270})
  \begin{equation}
  	\phi_F : Kh_t(K_0) \longrightarrow Kh_t(K_1)
  \end{equation}
by decomposing $F$  into a sequence of elementary cobordisms, each represented by a Reidemeister move or a Morse move of the link diagrams. The Reidemeister moves induce isomorphisms on homology. The 0-Morse move (corresonding to a birth), 1-Morse move (saddle), 2-Morse move (death) induce maps on $Kh_t$ according to the unit map, the multiplication/comultiplication map, the counit map, respectively. The map induced by $F$ is the composition of the maps induced by these elementary cobordisms. Furthermore, one can consider dotted cobordisms, where a dot induces a multiplication by $x$ on homology. Such assignment is well-defined and functorial, \textbf{up to a sign}.

	The cobordism maps satisfy the following local properties(see \cite{MR4092319}):
	
	(1) Suppose $F$ has an 1-handle $h$. Let $F'$ be obtained from $F$ by surgery along $h$, and let $\dot{F_1}$ and $\dot{F_2}$ be obtained by adding a dot to $F'$ at either of the feet of $h$. Then (see Fig.\ref{graph2}) up to a sign,
	\begin{equation}\label{p1}
		\phi_{F}= \phi_{\dot{F_1}} +\phi_{\dot{F_2}}
	\end{equation}

\begin{figure}[h]
	\def\svgwidth{\columnwidth}
	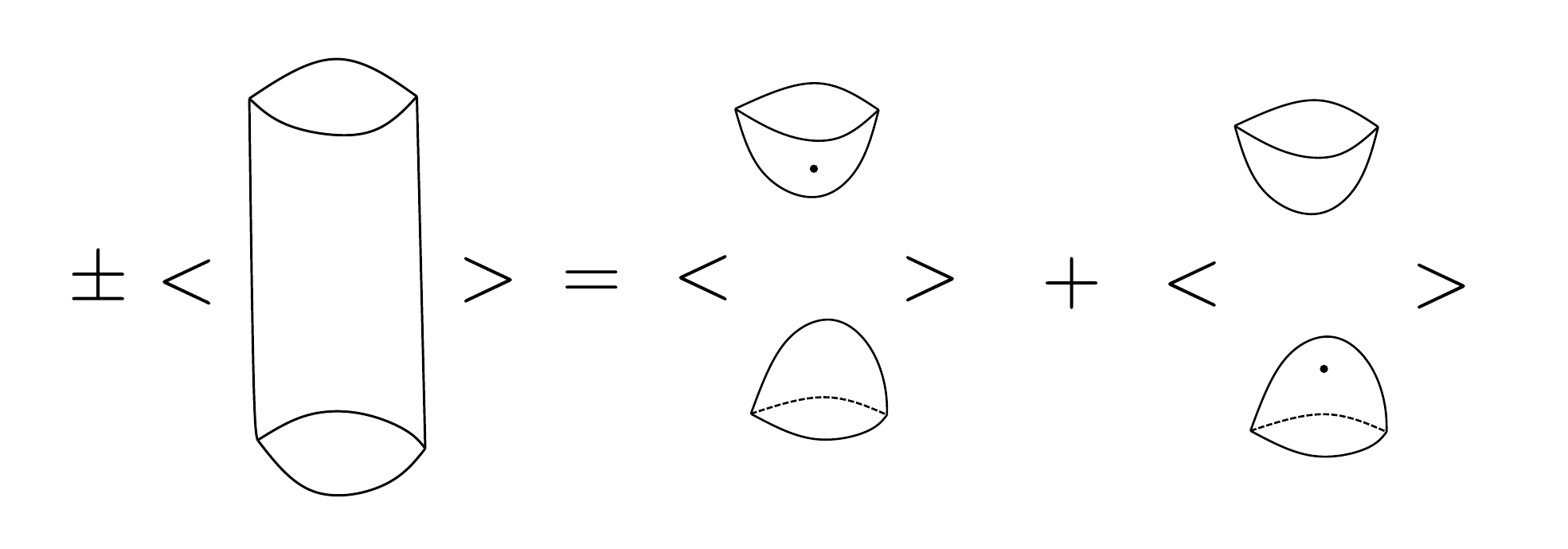
	\caption{} 
	\label{graph2}
\end{figure}

	(2) If $F$ is a cobordism which can be decomposed into a sequence $F_1,...F_n$ in which $F_i, F_{i+1}$ are two adjacent saddle moves reverse to each other, let $S$ be the cobordism obtained from $F$ by deleting $F_i,F_{i+1}$, and let $\dot{S}_1, \dot{S}_2$ be the cobordism obtained by adding a dot on one of the two sides of $S$(see Fig.\ref{graph3}). Then  up to a sign 
	\begin{equation}\label{p2}
		\phi_F =\phi_{\dot{S}_1}+ \phi_{\dot{S}_1}
	\end{equation}

	\begin{figure}[h]
		\def\svgwidth{\columnwidth}
		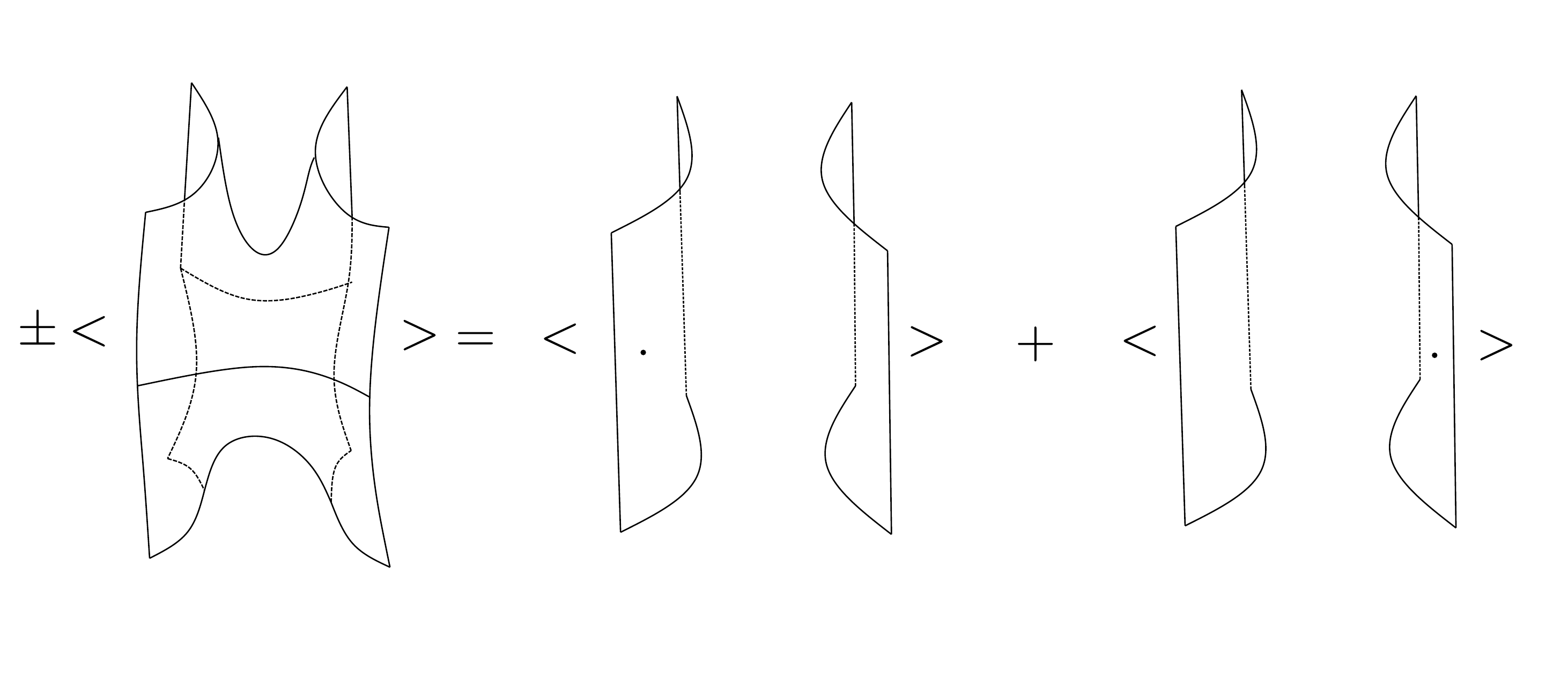
		\caption{} 
		\label{graph3}
	\end{figure}

	\section{The cobordism distance}
	
		\emph{Proof of Theorem \ref{thm1}}.\begin{proof}
	
	 Let $F$ be a connected oriented knot (dotted) cobordism from $K_0$ to $K_1$ in $S^3 \times [0,1]$ with $m$ births, $b$ saddles and $M$ deaths. We can rearrange the critical points of $F$ so that $F$ can be decomposed into a sequence of elementary cobordisms ordered as below(we omit the cobordisms of Reidemeister moves since they induce isomorphisms on homology):
	
	(1) $m$ births, which add $m$ unknots $U_1,...U_m$;
	
	(2) $m$ fusion saddle moves, which merge $U_1,...,U_m$ with $K_0$;
	
	(3) $b-m$ additional saddle moves;
	
	(4) $M$ deaths, deleting unknots $U'_1,...,U'_M$. \\
So the cobordism $S$ decomposes as $K_1 \stackrel{F_1}{\longrightarrow} K_1 \sqcup U^m \stackrel{F_2}{\longrightarrow} K' \stackrel{F_3}{\longrightarrow} K_2 \sqcup U^M \stackrel{F_4}{\longrightarrow} K_2$,
where the piece $F_i$ comes from item (i) above, and $U^m$ denotes the $m$- component planar unlink.
The cobordism $\bar{F}$ decomposes as $K_2 \stackrel{\bar{F_4}}{\longrightarrow} K_2 \sqcup U^M \stackrel{\bar{F_3}}{\longrightarrow} K' \stackrel{\bar{F_2}}{\longrightarrow} K_1 \sqcup U^m \stackrel{\bar{F_1}}{\longrightarrow} K_1$ . Then $\phi_{\bar{F}} \circ \phi_F$ is the composition of 8 maps. The composition of the fourth and fifth step is to delete the $M$ unknots via $M$ deaths, then add them back with $M$ births. Let $S$ be the cobordism obtained by deleting the forth and fifth step. Then $\bar{F} \circ F$ is obtained from $S$ by sugerying along the $M$ tubes. By   \ref{p1}  , we have
	\begin{equation}\label{eq1}
		\phi_S = (2x)^M \phi_{\bar{F} \circ F}
	\end{equation}
since $\bar{F} \circ F$ is connected.

Let $T$ be the cobordism obtained from $S$ by deleting the third and sixth steps, i.e. $T$ is the composition $K_1  \stackrel{F_1}{\longrightarrow} K_1 \sqcup U^m \stackrel{F_2}{\longrightarrow} K' \stackrel{\bar{F_2}}{\longrightarrow} K_1 \sqcup U^m \stackrel{\bar{F_1}}{\longrightarrow} K_1$. By   \ref{p2}  ,we have
\begin{equation}\label{eq2}
	\phi_S = (2x)^{b-m} \cdot \phi_{T}
\end{equation}
	Note that $F_2 \circ F_1$ is a ribbon concordance, by  \cite{MR4041014} Theorem 1 , 
	
	\begin{equation}\label{eq3}
		\phi_T = id_{Kh_t(K_0)}
	\end{equation}

	Theorem \ref{thm1} now follows from equation \ref{eq1} \ref{eq2} \ref{eq3}.
		
\end{proof}

	\emph{Proof of Corollary \ref{c1}} 
	\begin{proof}
	Let $F$ be an oriented knot cobordism from $K_0$ to $K_1$ with $m$ births, $b$ saddles, $M$ deaths. For $l \geq 0, \alpha \in Kh_t(K_0)$, by Theorem \ref{thm1} we have
	\begin{equation}
		\phi_{\bar{F} \circ F} \left( (2x)^{l+M} \cdot \alpha \right)  =(2x)^M \phi_{\bar{F} \circ F}\left( (2x)^{l} \cdot \alpha \right) =(2x)^{l+b-m} \cdot \alpha
	\end{equation}
	This shows that if $l \geq max\{0, xo(K_1)-M\}$, then \begin{equation}
		(2x)^{l+b-m} \cdot \alpha =\phi_{\bar{F}} \circ \phi_F \left( (2x)^{l+M} \cdot \alpha\right) =0
	\end{equation}
	
	Therefore $xo(K_0) \leq b-m+ max\{0, xo(K_1) -M\} = max\{M, xo(K_1)\} +2g$ (recall that we work over a base field in which 2 is invertible) since $2g(F)=-\chi(F)=b-m-M$.

	\end{proof}

	\emph{Proof of Corollary \ref{c2}}
	\begin{proof}
		By definition, after attaching $ul_b(K)$ oriented bands to $K$, we obtain an unlink $L$. Suppose $L$ has $M$ components. By performing $M-1$ deaths, we obtain a cobordism $S$ from $K$ to an unknot $U$ with $0$ births, $ul_B(K)$ saddles and $M-1$ deaths. Then by Corollary \ref{c1} we have
		\begin{equation}
			xo(K) \leq max \{M-1, 0\}+2g(S)= M-1+ul_b(K)-M+1 =ul_b(K).
		\end{equation}
	\end{proof}
	
	The following corollary generalizes \cite{MR4092319} Theorem 1.1:
	\begin{corollary}
		If $F$ is a knot concordance from $K_0$ to $K_1$ with $b$ saddles. Then
		\begin{equation}
			(2x)^bKh_t(K_0) \cong (2x)^b Kh_t(K_1)
		\end{equation}
	\end{corollary}
	\begin{proof}
		Suppose $F$ has $m$ births and $M$ deaths. Since $F$ is an annulus, $\chi(F)=m-b+M=0, M=b-m$. By Theorem \ref{thm1},
		\begin{equation}
			(2x)^M \phi_{\bar{F}} \circ \phi_{F} = (2x)^{b-m} \cdot Id_{Kh_t(K_0)} =(2x)^M Id_{Kh_t(K_0)}
		\end{equation}
	Note that $b \geq M$, for any $\alpha \in Kh_t(K_0)$,
	\begin{equation}
		\phi_{\bar{F}} \circ \phi_{F} \left((2x)^b \cdot \alpha \right)= (2x)^M  \phi_{\bar{F}} \circ \phi_{F} \left((2x)^{b-M} \cdot \alpha  \right)=(2x)^b \cdot \alpha     
	\end{equation}
	Therefore $\phi_{\bar{F}} \circ \phi_{F}$ is the identity map on $(2x)^b \cdot Kh_t(K_0)$.   Reversing the role of $F$ and $\bar{F}$, we get $\phi_{F} \circ \phi_{\bar{F}}$ is the identity map on $(2x)^b \cdot Kh_t(K_1)$. So $\phi_{F}: (2x)^b \cdot Kh_t(K_0) \longrightarrow (2x)^b \cdot Kh_t(K_1)$ is an isomorphism.

	\end{proof}

	\bibliography{ref}
\nocite{*}
\end{document}